\documentclass{amsart}
\usepackage{amsmath,amsthm,amssymb, enumerate, comment, xcolor}

\newcommand{\G}{\mathcal{G}}

\newcommand{\R}{\mathcal{R}}

\renewcommand{\phi}{\varphi}

\newcommand{\Prob}[1]{\operatorname{Prob}{#1}}

\newtheorem{theorem}{Theorem}[section]
\newtheorem{lemma}[theorem]{Lemma}
\newtheorem{proposition}[theorem]{Proposition}
\newtheorem{corollary}[theorem]{Corollary}

\newtheorem*{maintheorem}{Theorem}
\theoremstyle{remark}
\newtheorem*{remark}{Remark}
\theoremstyle{definition}
\newtheorem{example}[theorem]{Example}
\newtheorem*{definition}{Definition}
\title{Non-strong ergodicity of canonical actions of the Thompson groups}
\author{Ryoya Arimoto}
\address{The Polytechnic University of Japan, 2-32-1 Ogawa-nishimachi, Kodaira-shi, Tokyo, 187-0035, Japan}
\email{r-arimoto@uitec.ac.jp}
\subjclass{Primary 37A40; Secondary 20F38, 20L05}

\keywords{Thompson group, strongly ergodic, topological full group}

\begin{document}

\begin{abstract}
We show that the canonical actions of the Thompson group $V$ and its generalizations on the Cantor set are not strongly ergodic.
This implies that the associated crossed product von Neumann algebras are not full.
This also yields a non-embedding result for the Thompson groups.
\end{abstract}

\maketitle

\section{Introduction}
In 1965, R. J. Thompson introduced the groups $F, T, V$ in his unpublished note and showed that they are finitely presented, and $T, V$ are simple.
Since then, these groups have attracted considerable attention and have been investigated extensively.

In 2017, U. Haagerup and K. K. Olesen showed that $T, V$ are not inner amenable (\cite{HaagerupOlesen}).
This property, non-inner amenability, is closely related to the fullness of their group von Neumann algebras, which is a strong form of non-amenability.
By the result of Effros (\cite{Effros}), if a discrete ICC group is not inner amenable, then its group von Neumann algebra does not have property $\Gamma$.
The absence of property $\Gamma$ is known to be equivalent to the fullness of a von Neumann algebra (\cite[Corollary 3.8]{Connes}).
Thus, non-inner amenability of a group is a sufficient condition for fullness of its group von Neumann algebra, and hence the group von Neumann algebras $L(T), L(V)$ of $T, V$ are full.
The result of Haagerup and Olesen is extended to the Higman--Thompson groups, one of the generalizations of the Thompson group, by E. Bashwinger and M. C. B. Zaremsky (\cite{BashwingerZaremsky}).
Note that the Thompson group $F$ is shown to be inner amenable, and in addition, its group von Neumann algebra is a McDuff factor (see \cite{Jolissaint}).

Since the Thompson groups can be viewed as subgroups of the group of homeomorphisms of the Cantor set $X_2:= \{ 1, 2 \}^{\mathbb{N}}$, they have the canonical action on the Cantor set.
With respect to the probability measure $\mu _2:= (\frac{1}{2}\delta _1 + \frac{1}{2} \delta _2)^{\mathbb{N}}$ on $X_2$, the action of $V$ on the Cantor set is non-singular, that is, for every $s \in V$, a Borel subset $A \subset X_2$ satisfies $\mu _2 (A) > 0$ if and only if $\mu _2 (s^{-1}A) > 0$.
A non-singular action gives rise to a crossed product von Neumann algebra.
Considering that the group von Neumann algebra $L(V)$ is full, it is natural to ask whether the crossed product von Neumann algebra $L^{\infty}(X_2, \mu _2) \bar{\rtimes} V$ is full or not.
The action $V \curvearrowright (X_2 , \mu _2)$ is ergodic.
Moreover, we will show that this action is non-amenable in the sense of Zimmer, using the fact that the Thompson group $V$ contains a non-amenable subgroup whose elements have Radon--Nikodym derivatives equal to $1$ except on a small subset of the Cantor set.
Thus, $L^{\infty}(X_2, \mu _2) \bar{\rtimes} V$ is a non-amenable factor.
Observe that the action $V \curvearrowright (X_2, \mu _2)$ is not free.
This causes difficulties, for instance, in that the crossed product von Neumann algebra $L^{\infty}(X_2) \bar{\rtimes} V$ differs from the von Neumann algebra associated with the orbit equivalence relation $L(\R _{V \curvearrowright X_2})$.

Regarding the question whether the crossed product $L^{\infty}(X_2) \bar{\rtimes} V$ is full or not, we obtain the following.
The same conclusion also holds for some generalizations of the Thompson group.

\begin{maintheorem}[see Proposition \ref{HigmanThompsonalgebra}]
Let $V$ be the Thompson group and $V \curvearrowright X_2$ be the canonical action.
Then the associated von Neumann algebra $L^{\infty}(X_2, \mu _2) \bar{\rtimes} V$ is a non-amenable, non-full factor of type $\mathrm{III}_{1/2}$.
\end{maintheorem}

To prove this, we show that the action $V \curvearrowright X_2$ is not strongly ergodic.
A non-singular action of a discrete group on a probability space is called strongly ergodic if any almost invariant sequence of measurable sets is trivial (see section \ref{stronglyergodic} for a precise definition).
If the crossed product von Neumann algebra is full, then the action is necessarily strongly ergodic. 
Consequently, if the action is not strongly ergodic, then the associated crossed product is not full.
As the fullness of von Neumann algebras is a strong form of non-amenability, strong ergodicity also reflects non-amenability.
To see that the action $V \curvearrowright X_2$ is not strongly ergodic, we focus on its orbit equivalence relation $\R_{V \curvearrowright X_2}$ and show that it is amenable.
We can derive the amenability of $\R_{V \curvearrowright X_2}$ from the groupoid of germs $\G _2$ of $V \curvearrowright X_2$, which is known to give rise to the Cuntz algebra $\mathcal{O}_2$ and is known to be amenable.

More generally, we treat topological full groups of amenable groupoids.
Since the Thompson groups have remarkable properties, many generalizations of them have been considered.
One of the broad generalizations of the Thompson group $V$ is given by the topological full groups $[[ \mathcal{G} ]]$ of groupoids $\mathcal{G}$.
They act on the unit space $\mathcal{G}^{(0)}$ of the groupoid. 
Topological full groups include the Higman--Thompson groups (\cite{Higman}) and the Brin--Thompson groups (\cite{Brin}).
For instance, there is a groupoid $\mathcal{G}_{d, k}$ such that its topological full group $[[ \mathcal{G}_{d, k} ]]$ is isomorphic to the Higman--Thompson group $V_{d, k}$ and its action $[[ \mathcal{G} _{d, k} ]] \curvearrowright \mathcal{G}^{(0)}_{d, k}$ is also isomorphic to the canonical action $V_{d, k} \curvearrowright X_{d, k}$ of the Higman--Thompson group on the Cantor set (see Example \ref{Cuntzgrpd}).
These groups, the Thompson group $V$, the Higman--Thompson groups, and the Brin--Thompson groups, are known to arise from a topologically principal, amenable, ample groupoid.
From the operator algebraic perspective, groupoids can be considered as a generalization of dynamical systems, and a topological full group of a groupoid is regarded as a realization of the dynamics of the groupoid.
This class is known to provide a rich source of examples of groups and group actions.

We show that not only the canonical action of the Thompson group but also those of topological full groups of topologically principal, amenable, ample groupoids are not strongly ergodic.

\begin{maintheorem}[see Theorem \ref{main}]
Let $\mathcal{G}$ be a topologically principal, amenable, ample groupoid.
Then the canonical action $[[ \mathcal{G} ]] \curvearrowright \mathcal{G}^{(0)}$ is not strongly ergodic with respect to any quasi-invariant probability measure on $\mathcal{G}^{(0)}$. 
\end{maintheorem}

Consequently, the canonical actions of the Thompson group $V$, the Higman--Thompson groups, and the Brin--Thompson groups are not strongly ergodic, and their associated crossed product von Neumann algebras are not full.
This theorem also readily implies that if $\G$ is a topologically principal, amenable, ample groupoid, $\mu$ is a quasi-invariant measure on $\G^{(0)}$, and $\Gamma$ is a subgroup of $[[ \G ]]$, then the restricted action $\Gamma \curvearrowright (\G^{(0)}, \mu)$ is not strongly ergodic either.
For example, this shows that for any embedding $\iota$ of the free group $F_m$ $(m \geq 2)$ into $[[ \G ]]$, the restricted action $\iota (F_m) \curvearrowright (\G^{(0)}, \mu)$ is not isomorphic to the Bernoulli action $F_m \curvearrowright (\{ 0, 1 \} , \frac{1}{2}\delta _0 + \frac{1}{2}\delta _1)^{F_m}$.

In the proof of the latter theorem, we show that an equivalence relation associated with an amenable groupoid is amenable.
Since our proof is not constructive, we do not know a concrete description of a non-trivial almost invariant sequence, even for the Thompson group $V$.

\subsection*{Acknowledgements}
The author is deeply grateful to Yusuke Isono for his invaluable comments and discussions.
He would also like to thank Paul Jolissaint for bringing his attention to \cite{Jolissaint, Jolissaint2}.
This work was supported by JSPS KAKENHI Grant Number JP26K16998.

\section{Preliminaries}
\subsection{Groupoids and topological full groups}
A groupoid $\G$ is a set with a subset $\G^{(0)} \subset \G$, called the unit space of $\G$, source and range maps $s, r \colon \G \to \G^{(0)}$, and a multiplication
\[
\G^{(2)} := \{ (g, h) \in \G \times \G \mid s(g) = r(h) \} \ni (g, h) \mapsto gh \in \G
\]
such that 
\begin{itemize}
\item for all $x \in \G^{(0)}, s(x) = x = r(x)$,
\item for all $g \in \G, gs(g) = g = r(g)g$,
\item for all $(g, h) \in \G^{(2)}, s(gh)=s(h)\text{ and }r(gh)=r(g)$,
\item for $(g, h), (h, k) \in \G^{(2)}, (gh)k = g(hk)$, and
\item for all $g \in \G$, there exists $g^{-1} \in \G$ such that $g^{-1} g = s(g)$ and $g g^{-1} = r(g).$ 
\end{itemize}
Let $\G _x := s^{-1}(x)$, $\G ^x := r^{-1}(x)$, and $\G _x^x := \G _x \cap \G ^x = s^{-1}(x) \cap r^{-1}(x)$ for $x \in \G ^{(0)}$.

A topological groupoid is a groupoid equipped with a topology for which maps $(g, h) \mapsto gh$ and $g \mapsto g^{-1}$ are continuous.
In this article, all groupoids are assumed to be locally compact, Hausdorff, and with a compact unit space.
A topological groupoid $\G$ is said to be minimal if $\{ r(g) \in \G \mid g \in \G _x \} \subset \G^{(0)}$ is dense for all $x \in \G^{(0)}$.
A topological groupoid $\G$ is said to be topologically principal if the set $\{ x \in \G^{(0)} \mid \G _x^x = \{ x \} \}$ is dense in $\G ^{(0)}$.

A topological groupoid is said to be \'{e}tale if the source and range maps are local homeomorphisms.
An \'etale groupoid $\G$ is said to be ample if its unit space $\G^{(0)}$ is totally disconnected.
A compact open subset $U$ of an \'etale groupoid $\G$ is said to be a full bisection if $s|_U \colon U \to \G^{(0)}$ and $r|_U \colon U \to \G^{(0)}$ are homeomorphisms.

Let $\G$ be an ample groupoid.
The collection $[[ \G ]]$ of full bisections of $\G$ forms a group under the operations $UV := \{ uv \in \G \mid u \in U, v \in V \}$ and $U^{-1} := \{ u^{-1} \in \G \mid u \in U \}$.
This group is called the topological full group of $\G$.
For a topologically principal, ample groupoid $\G$, $[[ \G ]]$ faithfully acts on $\G^{(0)}$ by $U.x := r((s|_U)^{-1}(x))$ for $U \in [[ \G ]]$ and $x \in \G^{(0)}$.

The amenability of locally compact groupoids has been introduced by Renault.
We introduce the amenability of \'etale groupoids as in \cite{DelarocheRenault}.

\begin{definition}[\cite{DelarocheRenault}]
An \'etale groupoid $\G$ is said to be amenable if there is a net $((m_i^x)_{x \in \G^{(0)}})_i$, where $m_i^x \in \Prob{\G^x}$, such that $x \mapsto m_i^x(f)$ is continuous for every $f \in C_c (\G)$ and $\| g.m_i^{s(g)} - m_i^{r(g)} \|_1 \to 0$ uniformly on every compact subset of $\G$.
\end{definition}

Note that amenability of \'etale groupoids is equivalent to nuclearity of its reduced groupoid $\mathrm{C}^*$-algebra (see \cite[Theorem 5.6.18]{BrownOzawa}).

\begin{example}\label{Cuntzgrpd}
Let $d, k \geq 2$ be integers.
For a positive integer $n$, put $[n] = \{ 1 , \ldots , n \}$.
Let $X _d := [d]^{\mathbb{N}}$ and $X_{d, k} := [k] \times [d]^{\mathbb{N}}$ be the Cantor sets.
For $\mu \in [k] \times [d]^* := [k] \cup \bigcup_{p=1}^\infty [k] \times [d]^p$ and $\omega \in X_d$, $|\mu|$ denotes a length of $\mu$, $\mu \omega \in X_{d, k}$ denotes their concatenation, and $\mu X_d := \{ \mu \omega \mid \omega \in X_d \} \subset X_{d, k}$.

Let 
\[
\G _{d, k} := \{ (\nu \omega , |\nu| - |\mu|, \mu \omega) \in X_{d, k} \times \mathbb{Z} \times X_{d, k} \mid \mu, \nu \in [k] \times [d]^* , \omega \in X_d \}
\]
be a groupoid with structure maps
\begin{gather*}
s(\nu \omega , |\nu| - |\mu|, \mu \omega) = \mu \omega, \quad r(\nu \omega , |\nu| - |\mu|, \mu \omega) = \nu \omega, \\
(\nu \omega , |\nu| - |\mu|, \mu \omega)^{-1} = (\mu \omega, |\mu| - |\nu|, \nu \omega) \\
(\mu _3 \omega, |\mu_3|-|\mu_2|, \mu _2 \omega) (\mu _2 \omega, |\mu _2| - |\mu _1|, \mu _1 \omega) = (\mu _3 \omega, |\mu_3|-|\mu_1|, \mu _1 \omega).
\end{gather*}
This becomes an ample, topologically principal groupoid with a topology generated by $\{ (\nu \omega, | \nu | - | \mu |, \mu \omega) \in \G_{d, k} \mid \omega \in X_d \}$, where $\mu, \nu \in [k] \times [d]^*$.
Moreover, this groupoid is amenable since its reduced groupoid $\mathrm{C}^*$-algebra $\mathrm{C}^*_r(\G _{d, k})$ is isomorphic to a nuclear $\mathrm{C}^*$-algebra $\mathbb{M}_k (\mathbb{C}) \otimes \mathcal{O}_d$, where $\mathcal{O}_d$ is the Cuntz algebra of degree $d$.

The topological full group $[[ \G_{d, k} ]]$ of this groupoid $\G_{d, k}$ is known to be isomorphic to the Higman--Thompson group $V_{d, k}$ (see \cite[Remark 6.3 and Section 6.7.1]{Matui}).
Here, the Higman--Thompson group $V_{d, k}$ is a collection of homeomorphisms of $X_{d, k}$ of the following form: 
let $\mu _1, \ldots , \mu _n, \nu _1 , \ldots , \nu _n \in [k] \times [d]^*$ be finite words such that $\bigsqcup _{i=1}^n \mu _i X_d = X_{d, k} = \bigsqcup _{i=1}^n \nu _i X_d$.
Then $X_{d, k} \ni \mu _i \omega \mapsto \nu _i \omega \in X_{d, k}$ defines a homeomorphism of $X_{d, k}$.
This homeomorphism is denoted by 
$
\left(
\begin{smallmatrix} 
\mu_1 & \cdots & \mu _n \\
\nu _1 & \cdots & \nu _n 
\end{smallmatrix}
\right)
$.
The isomorphism is given by
\[
V_{d, k} \ni 
\left(
\begin{smallmatrix} 
\mu_1 & \cdots & \mu _n \\
\nu _1 & \cdots & \nu _n 
\end{smallmatrix}
\right)
\mapsto \bigsqcup _{i=1}^n \{ (\nu _i \omega , |\nu _i| - |\mu _i|, \mu _i \omega ) \mid \omega \in X_d \} \in [[ \G _{d, k} ]].
\]
\end{example}

\begin{example}\label{BrinThompson}
Using the similar notations as the above example, let
\[
\mathcal{G}_2 := \{ (\nu \omega , |\nu| - |\mu|, \mu \omega) \in X_2 \times \mathbb{Z} \times X_2 \mid \mu, \nu \in [2]^* , \omega \in X_2 \}.
\]
Then $\G _2^m := \overbrace{\G _2 \times \cdots \times \G _2}^{m}$ is an ample, topologically principal groupoid.
This groupoid is also amenable since $\mathrm{C}_r^*(\G_2^m)$ is isomorphic to $\mathcal{O}_2^{\otimes m}$.
It is well known that its topological full group $[[ \G _2^m ]]$ is isomorphic to the Brin--Thompson group $mV$.

\end{example}

\subsection{Equivalence relations}
An equivalence relation $\R$ on a set $X$ is a principal groupoid whose unit space is $X$, i.e., a groupoid with $\R _x \cap \R ^x = \{ x \}$. 
Let $(X, \mu)$ be the standard probability space.
A Borel equivalence relation $\R$ on $X$ is an equivalence relation on $X$ such that $\R \subset X \times X$ is a Borel subset.
Let $[x]_{\R}$ denote the equivalence class of $\R$ containing $x \in X$.
In this article, we assume that all equivalence classes of an equivalence relation are countable.
The equivalence relation $\R$ is said to be non-singular for $\mu$ if for each $E \subset X$ with $\mu (E)=0$, one has $\mu ( [E]_{\R})=0$, where $[E]_{\R} := \bigcup _{x \in E} [x]_{\R}$.
The full group $[ \R ]$ of $\R$ is a collection of Borel automorphisms $\phi \colon X \to X$ such that $\{ (\phi (x), x ) \in X \times X \mid x \in X \} \subset \R$.

\begin{example}
Assume that a discrete group $\Gamma$ acts on the standard Borel space $X$ by Borel automorphisms.
Its orbit equivalence relation $\R_{\Gamma \curvearrowright X}$ is an equivalence relation defined by $\{ (sx, x) \in X \times X \mid s \in \Gamma, x \in X \}$.
\end{example}

\begin{example}
Let $\G$ be an ample groupoid.
The associated equivalence relation $\R_\G$ is an equivalence relation on $\G^{(0)}$ defined by $\R_\G := \{ (r(g), s(g)) \in \G ^{(0)} \times \G^{(0)} \mid g \in \G \}$.

The groupoid $\G_{d, k}$ in Example \ref{Cuntzgrpd} gives the tail equivalence relation on $[k] \times [d]^{\mathbb{N}}$, i.e., for $x=(x_i)_{i=1}^{\infty}, y=(y_i)_{i=1}^{\infty} \in [k] \times [d]^{\mathbb{N}}$, $(x, y) \in \R_{\G_{d, k}}$ if there exist $p, q \in \mathbb{N}$ such that $x_{p+i} = y_{q+i}$ for all $i \geq 1$.
\end{example}

\begin{comment}
\begin{remark}
It is not clear whether $\R_\G = \R ([[ \G ]] \curvearrowright \G^{(0)})$.
\end{remark}
\end{comment}

The measure-theoretic notion of amenability for equivalence relations is introduced by Zimmer (\cite{Zimmer}), and an equivalent condition is given by Connes, Feldman, and Weiss (\cite{ConnesFeldmanWeiss}).
The Borel-theoretic notion of amenability for equivalence relations is introduced by Jackson, Kechris, and Louveau (\cite{JacksonKechrisLouveau}).

\begin{definition}[see \cite{Zimmer, ConnesFeldmanWeiss, JacksonKechrisLouveau}]
A countable Borel equivalence relation $\R$ on the standard probability space $(X, \mu)$ is said to be (Borel) amenable if there is a net $((p_i^x)_{x \in X})_i$, where $p_i^x \in \Prob \R^x$ such that $\R \ni (y, x) \mapsto p_i^y(y, x)$ is measurable and $\| \gamma p_i^x - p_i^y \| _1 \to 0$ for all $\gamma = (y, x) \in \R$, where $\gamma p_i^x (y, z):= p_i^x(x, z)$.
It is said to be $\mu$-amenable if there is a $\mu$-conull Borel subset $A \subset X$ such that $\R |_A$ is amenable.
\end{definition}

Note that a subequivalence relation of a $\mu$-amenable equivalence relation is $\mu$-amenable (see \cite{ConnesFeldmanWeiss} and \cite[Proposition 5.1]{Moore}).

With a non-singular equivalence relation $\R$ on a standard probability space $(X, \mu)$, one can associate a von Neumann algebra $L(\R)$.
This associated von Neumann algebra $L(\R)$ is generated by unitaries $\{ u_\phi \mid \phi \in [ \R ] \}$ and $L^{\infty}(X)$ with relations $u_{\phi} f u_{\phi}^* = f \circ \phi^{-1}$ for $\phi \in [\R]$ and $f \in L^{\infty}(X)$.
For an accurate definition of von Neumann algebras associated with equivalence relations, see, for example, \cite{FeldmanMoore}.
Note that $\R$ is $\mu$-amenable if and only if $L(\R)$ is amenable.

\begin{comment}
We review the definition of von Neumann algebras associated with equivalence relations.
Let $\R$ be a non-singular equivalence relation on a standard probability space $(X, \mu)$.
The full group $[ \R ]$ of $\R$ is a collection of Borel automorphisms $\phi \colon X \to X$ such that $\{ (\phi (x), x ) \in X \times X \mid x \in X \} \subset \R$.
The right counting measure $\mu _r$ on $\R$ is defined as follows:
\[
\mu _r (A) := \int _X |\R_y \cap A| \, d\mu(y)
\]
for a Borel set $A \subset \R$.
For each $\phi \in [\R]$, the associated unitary $u_\phi \in \mathbb{B}(L^2(\R, \mu_r))$ is defined by $[u_\phi \xi](x, y) := \xi (\phi^{-1}(x), y)$ for $\xi \in L^2 (\R, \mu _r)$ and $(x, y) \in \R$.
We also define a multiplication operator $M_f$ for $f \in L^\infty (X, \mu)$ by $[M_f \xi](x, y) := f(x) \xi (x, y)$ for $\xi \in L^2 (\R, \mu _r)$ and $(x, y) \in \R$.
The von Neumann algebra $L(\R)$ associated with $\R$ is a von Neumann algebra generated by $\{ M_f \mid f \in L^{\infty}(X, \mu) \}$ and $\{ u_{\phi} \mid \phi \in [ \R ] \}$ in $\mathbb{B}(L^2(\R, \mu _r))$.
\end{comment}

\subsection{Strongly ergodic actions and full factors}\label{stronglyergodic}
Let $\Gamma \curvearrowright (X, \mu)$ be a non-singular action of a discrete group $\Gamma$ on a standard probability space $(X, \mu)$.
The action $\Gamma \curvearrowright (X, \mu)$ is said to be strongly ergodic if for any sequence $(A_n)_{n \in \mathbb{N}}$ of measurable subsets of $X$ such that $\lim _n \mu (A_n \bigtriangleup s A_n) = 0$ for all $s \in \Gamma$, we have $\lim_n \mu (A_n) (1- \mu (A_n)) = 0$.
Note that any amenable action is never strongly ergodic.

Next, we review the fullness of factors.
Fullness of factors has been introduced by Connes.
He defines fullness by several equivalent conditions.
Here, we introduce fullness by using centralizing sequences.
For a von Neumann algebra $M$, $x \in M$ and $\varphi \in M_*$, define $x \varphi, \varphi x \in M_*$ by $[x \varphi](y) := \varphi (yx)$ and $[\varphi x] (y) := \varphi (xy)$.

\begin{definition}[\cite{Connes}]
A factor $M$ with separable predual is called full if for any bounded sequence $(x_n)_n$ of elements of $M$ such that $\| x_n \varphi - \varphi x_n \| \to 0$ for any $\varphi \in M_*$, there exists a bounded sequence $(z_n)_n$ of elements of $\mathbb{C}$ such that $x_n - z_n \to 0$ in the $*$-strong topology.
\end{definition}

If the action $\Gamma \curvearrowright (X, \mu)$ is not strongly ergodic, then the crossed product $L^{\infty} (X) \bar{\rtimes} \Gamma$ is never full.
Indeed, if $(A_n)_n$ is a non-trivial almost invariant sequence, i.e., $\lim _n \mu (A_n \bigtriangleup s A_n) = 0$ for all $s \in \Gamma$ and $\limsup_n \mu (A_n) (1- \mu (A_n)) \neq 0$, then $(1_{A_n})_n \subset L^{\infty}(X, \mu) \bar{\rtimes} \Gamma$ is a non-trivial centralizing sequence.
In some cases, strong ergodicity is one of the sufficient conditions for fullness of the crossed product, e.g., \cite{Choda, HoudayerIsono}.

\section{Main result}
\begin{lemma}\label{Lemma}
Let $\G$ be an amenable \'etale groupoid.
Then the associated equivalence relation $\R _\G$ is Borel amenable.
\end{lemma}

\begin{proof}
Let $(m_i^x)_{x \in \G^{(0)}}$ be a net of a family of maps as in the definition of amenability.
Define $\displaystyle p_i^x (x, y) := \sum _{s(g)=y, r(g)=x} m_i^x(g)$ for $x \in \G^{(0)}$, then $p_i^x \in \Prob \R_\G^x$ since
\begin{align*}
\sum_{y \in \G^{(0)}; (x, y) \in \R _\G} p_i^x (x, y) &= \sum _{y \in \G^{(0)}; (x, y) \in \R _\G} \sum_{s(g)=y, r(g)=x} m_i^x(g) \\
&= \sum _{g \in \G ^x} m_i^x (g) \\
&=1.
\end{align*}
We will show that $\| \gamma p_i^y - p_i^x \|_1 \to 0$ for $\gamma = (x, y) \in \R_\G$.
Fix $g \in \G$ with $s(g) = y, r(g) = x$.
Then one has
\begin{align*}
\gamma p_i^y (x, z) &= p_i^y (y, z) \\
&= \sum_{s(h)=z, r(h)=y} m_i^y(h) \\
&= \sum_{s(k)=z, r(k)=r(g)} gm_i^{s(g)}(k).
\end{align*}
One also has $\displaystyle p_i^x(x,z) = \sum_{s(k)=z, r(k)=r(g)} m_i^{r(g)}(k)$ and we have
\begin{align*}
\| \gamma p_i^y - p_i^x \|_1 &= \sum_{(x, z) \in \R_\G^x} \left| \sum_{s(h)=z, r(h)=x} gm_i^{s(g)}(h) - m_i^{r(g)}(h) \right| \\
& \leq \| gm_i^{s(g)} - m_i^{r(g)} \|_1 \\
&\to 0
\end{align*}
Therefore, $\R_\G$ is amenable.
\end{proof}

\begin{theorem}\label{main}
Let $\G$ be a topologically principal, amenable, ample groupoid.
Then the canonical action $[[ \G ]] \curvearrowright \G^{(0)}$ is not strongly ergodic with respect to any quasi-invariant probability measure $\mu$ on $\G^{(0)}$.
\end{theorem}

\begin{proof}
Since the equivalence relation $\R_\G$ associated to $\G$ is $\mu$-amenable by Lemma \ref{Lemma}, by Rokhlin's lemma (\cite[Proposition 2.2]{Schmidt}), there is a sequence $(A_n)$ of measurable subsets of $\mathcal{G}^{(0)}$ such that $\mu (A_n \triangle s A_n) \to 0$ for all $s \in [\R]$ and $\mu (A_n) (1-\mu (A_n)) \not\to 0$.
Since $[[ \G ]] < [ \R_\G ]$, this $(A_n)$ is a non-trivial almost invariant sequence for $[[ \G ]] \curvearrowright (\G^{(0)}, \mu)$.
Therefore, $[[ \G ]] \curvearrowright (\G^{(0)}, \mu)$ is not strongly ergodic.
\begin{comment}
Using the fact that $[[ \G ]] < [ \R _\G ]$, one has $(L^{\infty}(X, \mu)^\omega)^{[[ \G ]]} \supset L^{\infty}(X, \mu) ^\omega \cap L (\R_\G)' $.
Thus, we have $(L^{\infty}(X, \mu)^\omega)^{[[ \G ]]} \neq \mathbb{C}$ and this shows that the action $[[ \G ]] \curvearrowright \G^{(0)}$ is not strongly ergodic.
\end{comment}
\end{proof}

As a corollary, we obtain a non-embedding result.
Note that there are many ways to embed free groups into $[[ \G ]]$.
For instance, it is not hard to see that there is an embedding $\iota \colon F_2 \hookrightarrow V$ such that the restricted action $\iota (F_2) \curvearrowright \{ 1, 2 \} ^{\mathbb{N}}$ is (topologically) isomorphic to the action $F_2 \curvearrowright \partial F_2$ of the free group on its Gromov boundary.

\begin{corollary}
Let $\G$ be a topologically principal, amenable, ample groupoid, $\mu$ a quasi-invariant measure on $\G^{(0)}$, and $\Gamma \curvearrowright (X, \nu)$ a strongly ergodic action on a standard probability space.
Then there is no embedding $\iota \colon \Gamma \hookrightarrow [[ \G ]]$ such that the restricted action $\iota (\Gamma) \curvearrowright (\G^{(0)}, \mu )$ is isomorphic to $\Gamma \curvearrowright (X, \nu)$.
\end{corollary}

In particular, for any embedding $\iota \colon F_m \hookrightarrow [[ \G ]]$, the restricted action $\iota (F_m) \curvearrowright (\G^{(0)}, \mu)$ is not isomorphic to the Bernoulli action $F_m \curvearrowright (\{ 0, 1 \} , \frac{1}{2}\delta _0 + \frac{1}{2}\delta _1)^{F_m}$, which is known to be strongly ergodic. 

Finally, we investigate the canonical action $V_{d, k} \curvearrowright X_{d, k}$ of the Higman--Thompson group on the Cantor set and the associated crossed product von Neumann algebra.
We will use the following theorem to see the non-amenability of the action.

\begin{theorem}[{\cite[Theorem 7]{DouglasNowak}, \cite[Proposition 5.3]{VaesWahl}}] \label{non-amenable}
Let $\Gamma$ be a countable group and $\Gamma \curvearrowright (X, \mu)$ be its non-singular action on a standard probability space.
Denote by $\omega \colon \Gamma \times X \to (0, + \infty)$ the Radon--Nikodym cocycle: $\displaystyle \omega (s, x) := \frac{ds^{-1}\mu}{d\mu}(x)$.
Assume that there exists a finite subset $F \subset \Gamma$ such that 
\[
\sum _{s \in F} \int _X \sqrt{\omega(s, x)} \, d\mu(x) > \left\| \sum _{s \in F} \lambda _s \right\|,
\]
where $\lambda \colon \Gamma \to \mathbb{B}(\ell ^2 (\Gamma))$ denotes the left regular representation.
Then the action $\Gamma \curvearrowright (X, \mu)$ is non-amenable in the sense of Zimmer.
\end{theorem}

Let $\mu_{d, k}$ be the probability measure on $X_{d, k} = [k] \times [d]^{\mathbb{N}}$ defined by $\mu_{d, k} := \left( \frac{1}{k} \sum_{i=1}^k \delta _i \right) \otimes \left( \frac{1}{d} \sum_{i=1}^d \delta _i \right)^{\mathbb{N}}$.

\begin{proposition}\label{HigmanThompsonalgebra}
Let $V_{d, k}$ be the Higman--Thompson group, $V_{d, k} \curvearrowright (X_{d, k}, \mu_{d, k})$ be the canonical action. 
Then the von Neumann algebra $L^{\infty}(X_{d, k}, \mu_{d, k}) \bar{\rtimes} V_{d, k}$ is a non-amenable, non-full factor of type $\mathrm{III}_{1/d}$.
\end{proposition}

\begin{proof}
First, we show that $L^{\infty}(X_{d, k}) \bar{\rtimes} V_{d, k}$ is a factor.
Let $(V_{d, k})_{\mu_{d, k}}$ be the subgroup of $V_{d, k}$ consisting of elements of $V_{d, k}$ which preserve $\mu_{d, k}$.
We will show that the restricted action $(V_{d, k})_{\mu_{d, k}} \curvearrowright (X_{d, k}, \mu_{d, k})$ is ergodic (cf. \cite{Jolissaint2}).
Let $A \subset X_{d, k}$ be a $(V_{d, k})_{\mu_{d, k}}$-invariant measurable subset with positive measure.
We claim that the value $\displaystyle c(\nu) := \frac{\mu_{d, k} (A \cap \nu X_d)}{\mu_{d, k} (\nu X_d)}$ is independent of $\nu \in [k] \times [d]^*$.
If $| \nu _1 | = | \nu _2|$, then there exists $s \in (V_{d, k})_{\mu_{d, k}}$ such that $s(\nu_1 X_d) = \nu_2 X_d$ and it follows that $\mu_{d, k} (A \cap \nu X_d)$ depends only on $| \nu |$.
Thus, for each $\nu \in [k] \times [d]^*$, 
\[
dc(\nu) \mu_{d, k} (\nu 1 X_d) = c(\nu) \mu_{d, k}(\nu X_d) = \mu_{d, k} (A \cap \nu X_d) = d \mu_{d, k} (A \cap \nu 1 X_d).
\]
Hence, one has $c(\nu) = c(\nu 1)$, and $c( \nu )$ is independent of $\nu$.
Let this value be $c$.
Then, for any disjoint union $B$ of finitely many cylinder sets, $\displaystyle \frac{\mu_{d, k}(A \cap B)}{\mu_{d, k} (B)} = c$.
Since the set of Borel sets which can be approximated by a disjoint union of finitely many cylinder sets is a $\sigma$-algebra of whole Borel sets, $\displaystyle \frac{\mu_{d, k}(A \cap B)}{\mu_{d, k} (B)} = c$ for every Borel subset $B$ of $X_{d, k}$.
Hence, we have $\displaystyle \mu _{d, k} (A) = c = \frac{\mu_{d, k} (A)}{\mu _{d, k} (A)} = 1$.

It is not hard to see that $V_{d, k}$ is ICC relative to $(V_{d, k})_{\mu_{d, k}}$, i.e., $\{ tst^{-1} \mid t \in (V_{d, k})_{\mu_{d, k}} \}$ is infinite for all $s \in V_{d, k} \setminus \{ e \}$.
Hence, the inclusion $L^{\infty}(X_{d, k}) \bar{\rtimes} (V_{d, k})_{\mu_{d, k}} \subset L^{\infty}(X_{d, k}) \bar{\rtimes} V_{d, k}$ is irreducible, i.e., $(L^{\infty}(X_{d, k}) \bar{\rtimes} (V_{d, k})_{\mu_{d, k}})' \cap L^{\infty}(X_{d, k}) \bar{\rtimes} V_{d, k} = \mathbb{C}$.
In particular, $L^{\infty}(X_{d, k}) \bar{\rtimes} V_{d, k}$ is a factor.

We next show that the crossed product $L^{\infty}(X_{d, k}, \mu_{d, k}) \bar{\rtimes} V_{d, k}$ is non-amenable.
By \cite{Zimmer2}, if a non-singular action $\Gamma \curvearrowright (X, \mu)$ is non-amenable in the sense of Zimmer, then the crossed product $L^{\infty}(X) \bar{\rtimes} \Gamma$ is non-amenable.
Thus, it suffices to show that the action $V_{d, k} \curvearrowright (X_{d, k}, \mu_{d, k})$ is non-amenable.
Since the Higman--Thompson group $V_{d, d}$ is non-amenable, there exists a finite subset $F \subset V_{d, d}$ such that $\displaystyle \| \sum _{s \in F} \lambda _s \| < |F|$ (see \cite[Theorem 2.6.8]{BrownOzawa}).
Take an $n \in \mathbb{N}$ such that $\displaystyle |F| \left( 1 - \frac{1}{kd^{n-1}} \right) > \| \sum _{s \in F} \lambda _s \|$.
Let $\nu = 1 \cdots 1 \in [k] \times [d]^{n-1}$ be a finite word consisting of $n$ consecutive 1s.
Consider an embedding $V_{d, d}$ into $V_{d, k}$ in such a way that $V_{d, d}$ acts canonically on $\nu X_d \cong X_d$  and trivially on $X_{d, k} \setminus \nu X_d$.
Note that for each $s \in V_{d, d} \subset V_{d, k}$ and $x \in X_{d, k} \setminus \nu X_d$, $\omega (s , x) =1$.
Thus, we have
\[
\sum _{s \in F} \int _{X_{d, k}} \sqrt{\omega (s, x)} \, d\mu_{d, k}(x) > \sum _{s \in F} \mu_{d, k} (X_{d, k} \setminus \nu X_d) = |F| \left( 1 - \frac{1}{kd^{n-1}} \right) > \| \sum _{s \in F} \lambda _s \|.
\]
Therefore, the action $V_{d, k} \curvearrowright (X_{d, k}, \mu_{d, k})$ is non-amenable by Theorem \ref{non-amenable}.
Note that this argument also applies to the topological full group of an essentially principal, ample, purely infinite, and minimal groupoid $\mathcal{G}$, hence the action $[[ \mathcal{G} ]] \curvearrowright \mathcal{G}^{(0)}$ is non-amenable.
For an embedding of $V$ into $[[ \mathcal{G} ]]$, see \cite[Proposition 3.7]{GardellaTanner}.

By Theorem \ref{main}, $L^{\infty} (X_{d, k}) \bar{\rtimes} V_{d, k}$ is not full.

Finally, we compute the type of $L^{\infty} (X_{d, k}) \bar{\rtimes} V_{d, k}$ (cf. \cite[Proposition 2.7]{Morando}).
Observe that the modular operator $\Delta_\varphi$ on $L^2 (X_{d, k}) \otimes \ell^2 (V_{d, k})$ of the canonical faithful state $\varphi (f \lambda _s) := \delta _{s, e} \mu_{d, k} (f)$ is given by $\displaystyle \Delta_\varphi (\xi \otimes \delta _s) := \left( \frac{ds \mu_{d, k}}{d \mu_{d, k}} \right)^{-1} \xi \otimes \delta _s$ and the modular automorphism group $(\sigma _t^{\varphi})_t$ on $L^{\infty}(X_{d, k}) \bar{\rtimes} V_{d, k}$ of $\varphi$ is given by $\displaystyle \sigma _t^{\varphi}(f \lambda_s) = \left( \frac{ds\mu_{d, k}}{d\mu_{d, k}} \right)^{-it} f \lambda _s$. 
Thus, we have 
\[
L^{\infty} (X_{d, k}) \bar{\rtimes} (V_{d, k})_{\mu _{d, k}} \subset (L^{\infty} (X_{d, k}) \bar{\rtimes} (V_{d, k}))^{\sigma ^\varphi} \subset L^{\infty} (X_{d, k}) \bar{\rtimes} V_{d, k},
\]
where $(L^{\infty} (X_{d, k}) \bar{\rtimes} (V_{d, k}))^{\sigma ^\varphi} := \{ x \in L^{\infty} (X_{d, k}) \bar{\rtimes} (V_{d, k}) \mid \sigma_t^{\varphi}(x) = x \text{ for all } t \in \mathbb{R} \}$, and $(L^{\infty} (X_{d, k}) \bar{\rtimes} (V_{d, k}))^{\sigma ^\varphi}$ is a factor.
Hence, by \cite[Corollary 3.2.7]{ConnesSinv}, Connes' $S$-invariant $S(L^{\infty} (X_{d, k}) \bar{\rtimes} V_{d, k})$ is $\sigma (\Delta _\varphi) = d^{\mathbb{Z}} \cup \{ 0 \}$ and thus $L^{\infty} (X_{d, k}) \bar{\rtimes} V_{d, k}$ is of type $\mathrm{III}_{1/d}$.
\end{proof}

\end{document}